\documentclass[12pt,twoside]{article}

\usepackage[a4paper]{geometry}
\usepackage[utf8]{inputenc}
\makeatletter
\renewcommand\title[1]{\gdef\@title{\reset@font\Large\bfseries #1}}
\renewcommand\section{\@startsection {section}{1}{\z@}%
                                   {-3.5ex \@plus -1ex \@minus -.2ex}%
                                   {2.3ex \@plus.2ex}%
                                   {\normalfont\large\bfseries}}
\renewcommand\subsection{\@startsection{subsection}{2}{\z@}%
                                     {-3ex\@plus -1ex \@minus -.2ex}%
                                     {1.5ex \@plus .2ex}%
                                     {\normalfont\normalsize\bfseries}}
\renewcommand\subsubsection{\@startsection{subsubsection}{3}{\z@}%
                                     {-2.5ex\@plus -1ex \@minus -.2ex}%
                                     {1.5ex \@plus .2ex}%
                                     {\normalfont\normalsize\bfseries}}

\def\@runningauthor{}\newcommand{\runningauthor}[1]{\def\runningauthor{#1}}
\def\@runningtitle{}\newcommand{\runningtitle}[1]{\def\runningtitle{#1}}

\renewcommand{\ps@plain}{%
\renewcommand{\@evenhead}{\footnotesize\scshape \hfill\runningauthor\hfill}
\renewcommand{\@oddhead}{\footnotesize\scshape \hfill\runningtitle\hfill}}
\g@addto@macro\bfseries{\boldmath}

\makeatother

\usepackage{amsthm,amsmath,amssymb}

\usepackage{graphicx}

\usepackage[colorlinks=true,citecolor=black,linkcolor=black,urlcolor=blue]{hyperref}

\theoremstyle{plain}
\newtheorem{theorem}{Theorem}
\newtheorem{lemma}[theorem]{Lemma}

\theoremstyle{definition}

\theoremstyle{remark}

\newcommand\F{{\mathbb F}}

\newcommand\N{{\mathbb N}}

\newcommand\Lm{\mathrm{Lm}}

\title{On the linear complexity for  multidimensional sequences}

\runningtitle{Linear complexity for multidimensional sequences}

\author{Domingo G\'omez-P\'erez\\
\small Department of Mathematics, Statistics and Computer Science\\[-0.8ex]
\small University of Cantabria\\[-0.8ex]
\small 39005 Santander, Spain\\
\small\tt domingo.gomez@unican.es\\
\and
Min Sha \\
\small Department of Computing\\[-0.8ex]
\small Macquarie University\\[-0.8ex]
\small Sydney, NSW 2109, Australia\\
\small\tt shamin2010@gmail.com \\
\and
Andrew Tirkel \\
\small Scientific Technology Pty Ltd.\\[-0.8ex]
\small 8 Cecil St, East Brighton\\[-0.8ex]
\small VIC 3187, Australia\\
\small\tt atirkel@bigpond.net.au
}

\runningauthor{D. G\'omez-P\'erez, M. Sha, A. Tirkel}

\date{}

\begin{document}

\maketitle

\thispagestyle{empty}

\begin{abstract}
In this paper, we define the linear complexity for multidimensional
sequences over finite fields, generalizing the one-dimensional case. 
We give some lower and upper bounds, valid with large probability, for the
linear complexity and $k$-error linear complexity of multidimensional
periodic sequences.
\end{abstract}




\section{Introduction}

One-dimensional periodic sequences with low auto- and
cross-correlations have extensive applications in modern communications.
Meanwhile, digital watermarking, which has been used to provide
copyright protection, certificates of authenticity, access control,
audit trail and many other security features, require 
multidimensional arrays (identified with multidimensional periodic sequences) with similar properties.
There are several constructions of these objects
proposed by Oscar Moreno, Andrew Tirkel et al.
\cite{BMT,MT0,MT1,TH}. 

Recently, in  \cite{GHMR}  the concept of linear
complexity of one-dimensional periodic sequences has been extended to
higher dimensions, and an efficient algorithm has been given. 
Moreover, the numerical results in \cite{GHMR} suggest that the
Moreno--Tirkel arrays \cite{MT1} have high linear complexity. 
This concept in fact is equivalent to ours for periodic sequences, which is explained later on. 

A cryptographically strong sequence should have a
high linear complexity, and it should also not be possible to
decrease significantly the linear complexity by changing a few terms of the sequence. 
This leads to the concept of $k$-error linear complexity defined by Stamp and Martin~\cite{SM},
which is based on the sphere complexity due to Ding, Xiao, and Shan~\cite{Ding}.
Note that, in practice, changes in the bitstream can occur due to noise,
multipath, or other distortion in the wireless channel.

In this paper, continuing previous work \cite{GHMR},  we define the
linear complexity for multidimensional sequences, including that of
multidimensional arrays as a particular example and introduce the
$k$-error linear complexity for such sequences.  We obtain 
some lower and upper bounds, valid with large probability, for the linear complexity and
$k$-error linear complexity of periodic sequences. 

The paper is organized as follows: Section~\ref{sec:basic} recalls
some basic definitions. 
The proofs of the main results
are based on some combinatorial analysis, which is included in Section \ref{sec:count}.
The main results are presented and proved in Section \ref{sec:lower}.

\section{Preliminaries}
\label{sec:basic}

\subsection{Multidimensional sequences}
\label{sec:seq}

Let $\N_0$ be the set of non-negative integers and $\F_q$ the finite
field of $q$ elements.
For any integer $n\ge 1$, an \textit{$n$-dimensional sequence} over $\F_q$ is a
mapping $s:$ $\N_0^n \to \F_q$. We write
$\pmb{m}=(m_1,\ldots,m_n)$ for the elements of $\N_0^n$, and the corresponding 
term in the sequence $s$ is denoted by $s(\pmb{m})$.
Further, let $\F_q[X_1,\ldots,X_n]$ be the polynomial ring in
variables $X_1,\ldots,X_n$ over $\F_q$.
A \textit{monomial} in this ring has the form
$$
\pmb{X^j} = X_1^{j_1} \ldots X_n^{j_n},
$$
where $\pmb{X}=(X_1,\ldots,X_n)$ and $\pmb{j}=(j_1,\ldots,j_n) \in \N_0^n$.

Let $\F_q[X_1,\ldots,X_n]$ act on the sequence $s$ as follows.
For any
$$
P(\pmb{X}) = \sum_{\pmb{j}} a_{\pmb{j}} \pmb{X^j} \in \F_q[X_1,\ldots,X_n],
$$
let $Ps$ be the $n$-dimensional sequence defined by 
$$
Ps(\pmb{m}) = \sum_{\pmb{j}} a_{\pmb{j}} s(\pmb{m}+\pmb{j}).
$$
We denote by $I(s)$ the set of polynomials $P\in \F_q[X_1,\ldots,X_n]$ for which
$Ps=0$. Clearly, each polynomial in $I(s)$ actually represents a
linear recurrence of $s$.
In fact, $I(s)$ is an ideal of the ring $\F_q[X_1,\ldots,X_n]$, so the quotient
$\F_q[X_1,\ldots,X_n]/I(s)$ is well defined (and is an $\F_q$-linear
space).
If the quotient space $\F_q[X_1,\ldots,X_n]/I(s)$ has finite dimension
(say $d$) over $\F_q$, we say that the sequence $s$ is an \textit{$n$-dimensional
  recurrence sequence of order $d$}.
We refer to the survey by Schmidt~\cite{Schmidt} for a general
introduction to this topic. When $n=1$, this definition recovers the
so-called \textit{linear recurrence sequence}; see the book by Everest
et al. \cite{EPSW} for an extensive introduction.
Moreover, for any ideal $I$, the quotient space
$\F_q[X_1,\ldots,X_n]/I$ has finite dimension over $\F_q$  if and only if
there is a non-zero polynomial in $I \cap \F_q[X_i]$ for each $i=1,\ldots,n$.

Particularly, the sequence $s$ is said to be \textit{periodic} if there is an $n$-tuple
 $(T_1,\ldots,T_n)$ of positive integers such that all the
binomials $X_1^{T_1}-1,\ldots,X_n^{T_n}-1$ belong to
$I(s)$, that is, the sequence is periodic in every  dimension.
Then, we call $(T_1,\ldots,T_n)$ a \textit{period} of $s$.
Periodic sequences of dimension two are called
  \textit{doubly-periodic sequences}, a largely studied object with applications in
algebraic coding theory~\cite{GH,Sakata0}.

An $n$-dimensional array $A$ of size $T_1 \times \cdots \times
T_n$ can be naturally extended to an $n$-dimensional sequence:
$$
s_A(m_1,\ldots,m_n) =
\textrm{$A(m_1 $\,mod $T_1,\ldots, m_n $\,mod  $T_n)$}.
$$
(Note that $(T_1,\ldots,T_n)$ is a period of $s_A$.)
Conversely, we can view every periodic sequence as the extension of an
array. Hence, we can identify multidimensional arrays with 
multidimensional periodic sequences.

The concept of multidimensional sequences we deal with
must not be confused with that of \textit{multisequences}, which
consists of finitely many parallel streams of one-dimensional
sequences \cite{MNV}.

\subsection{Linear complexity}
\label{sec:Lin com}

Recall that, in dimension one, the \textit{linear complexity} of
  a periodic sequence coincides with its order.
Similarly, we define the \textit{linear complexity} of a
multidimensional sequence $s$ to be its order (as defined above), denoted by $L(s)$.
So, the only sequence with linear complexity equal to zero is the zero sequence.
The \textit{linear complexity} of an $n$-dimensional array $A$ is
defined as the linear complexity of its periodic extension $s_A.$

A definition of linear complexity for multidimensional arrays (identified with periodic sequences) has been employed in \cite{GHMR} to test
the security of some multidimensional arrays 
proposed by Moreno and Tirkel \cite{MT1},
which is in fact equivalent to our definition.
We remark that the definition we give above is more formal and more
general than that in \cite{GHMR},
because it is a purely algebraic definition and it 
does not need to assume that the sequence is periodic. 
However, it is easier to design an algorithm
building on the definition in \cite{GHMR}: the cardinality of a
Delta set, as explained below.

Fix a monomial order $<_{\tau}$ on
$\F_q[X_1,\ldots, X_n]$.  The
maximum term of $f$ with respect to $<_{\tau}$ of a polynomial
$f$ is called the
\textit{leading term}, and the corresponding
monomial, denoted by $\Lm(f)$, is the \textit{leading monomial} of $f$.

A  \textit{Gr{\"o}bner basis} of a non-zero ideal $I$ of
$\F_q[X_1,\ldots,X_n]$ (with respect to $<_\tau$) is a
polynomial set $G(I)=\{g_1,\ldots, g_m\}$ which generate the ideal $I$
and  such that,
if $f\in I$, there is a polynomial in $G(I)$ whose leading monomial
divides $\Lm(f)$.
The basis $G(I)$ is said to be \textit{reduced} if, for each $i=1,\ldots, m$, the polynomial $g_i$ is monic
and its leading monomial does not divide any non-zero term of the other polynomials in $G(I)$.
It is well-known that there is exactly one reduced Gr{\"o}bner basis of $I$ with respect to
$<_{\tau}$.

Write 
$\Lm(I) = \{\Lm(f)\;: \; f\in I\}$, and
define the \textit{Delta set} of $I$ as
$$
\Delta(I) = \{\pmb{j}\in \N_0^n:\, \pmb{X^j} \not\in \Lm(I) \}.
$$
Note that $\pmb{X^j} \not\in \Lm(I)$ if and only if $\pmb{X^j}$
is not divisible by the leading monomial of any polynomial in $G(I)$.
The Delta
set depends on the chosen monomial order  $<_{\tau}$, but the cardinality of
the Delta set does not depend on $<_{\tau}$ and thus is an ideal invariant.
We refer to the book by Cox, Little, and O' Shea \cite{Cox} for
  more details.

Let $s$ be a periodic sequence with period
$(T_1,\ldots, T_n)$. Then,
the polynomials $X_1^{T_1}-1,\ldots, X_n^{T_n}-1$
are elements of the ideal $I(s)$, and
so its Delta set, denoted by $\Delta(s)$, is finite.
Moreover, the \textit{linear complexity} of the sequence satisfies 
\begin{equation}    \label{eq:LD}
L(s) = | \Delta(s)|\le T_1T_2 \cdots T_n.
\end{equation}
Furthermore, $L(s)$ is the minimum number of initial terms
which generate the whole sequence through the linear recurrences
represented by the polynomials in $I(s)$
(or equivalently, the polynomials in $G(s)$).
Such initial terms are exactly:
$$
s(\pmb{j}), \quad \pmb{j} \in \Delta(s).
$$

\subsection{Remarks on linear complexity}
\label{sub:2_3}
In the literature, there is another definition of the 
linear complexity for two dimensional binary finite sequences,
which is proposed by Gyarmati, Mauduit, and Sárközy \cite[Definition 5]{GMS1} 
and is equal to the minimal number
of initial terms generating the whole sequence by a specific linear recurrence. 
We opt instead to define the linear complexity of a
finite multidimensional sequence as that of its periodic extension. 
Note that a linear recurrence of a finite multidimensional sequence may be 
not a linear recurrence of its periodic extension.

Besides, if
$T_1,\ldots,T_n$ are pairwise coprime, any $n$-dimensional sequence
$s$ of period $(T_1,\ldots,T_n)$ can be converted 
into a one-dimensional sequence $t$ of period $T_1T_2\cdots T_n$ by
$$
t(m) = \textrm{$s(m$ mod $T_1, \ldots, m$ mod $T_n)$}, \quad m \ge 0.
$$
Since $T_1,\ldots,T_n$ are pairwise coprime, there is a one-to-one correspondence between
the terms of $s$ and $t$. So, the shortest linear recurrence generating $t$ can be converted into
a linear recurrence which  generates $s$. Hence,
$$
L(s) \le L(t).
$$

\subsection{The $k$-error linear complexity}

Let $s$ be an $n$-dimensional periodic sequence  with period
$(T_1,\ldots,T_n)$. 
Given an integer $k\le T_1\cdots T_n$, 
the \textit{$k$-error linear complexity} $L_k(s)$ of $s$ is the
smallest linear complexity among those sequences which differ from $s$
in $k$ or fewer terms from a period:
$$
\{(m_1,\ldots,m_n): \,  0 \le m_i \le T_i-1, 1\le i \le n\}.
$$
It follows from the definition that
$$
0=L_{T_1\cdots T_n}(s) \le \cdots \le L_2(s) \le L_1(s) \le
L_0(s)=L(s) \le T_1\cdots T_n.
$$


\section{Some counting results}
\label{sec:count}

In this section, we establish a couple of results regarding the amount
of monomial ideals in a polynomial ring and the size of a reduced Gr{\" o}bner
basis, which are used later on.
Here, $\F$  stands for an arbitrary field.

Recall that an ideal of the polynomial ring $\F[X_1,\ldots,X_n]$ is  a
\textit{monomial ideal} if it is generated by monomials.
A monomial lies in an ideal generated by some monomials
if and only if it is divisible by one of
them. Besides, a polynomial lies in a monomial ideal if
and only if all its monomials do; see \cite[Chapter 2, \S 4]{Cox}.

\begin{lemma}  \label{lem:monomial}
For any positive integer $K \ge 1$,
let $M_n(K)$ be  the set of monomial ideals $I$ in
$\F[X_1,\ldots,X_n]$ such that the dimension of the quotient space
$\F[X_1,\ldots,X_n] / I$ over $\F$  equals $K$. Then, 
$|M_n(1)|=1$, and for any $K\ge 2$, 
$$
|M_n(K)| \le K^n(K+n-2)^{(n-1)(K-1)} (2K-3)^{K-2}.
$$
In particular, for any real number $R >1$, there exists a
constant $c$ depending on $n$ and $R$ such that for any $K \ge 1$,
$$
|M_n(K)| \le cR^{K^{3/2}}.
$$
\end{lemma}

\begin{proof}
Since the set $M_n(1)$ only contains the ideal generated by $X_1,\ldots, X_n$, we indeed have $|M_n(1)|=1$. 
In the following, we assume that $K \ge 2$. 

Given a monomial ideal $I \in M_n(K)$, 
we must have that  for each $i=1,\ldots, n$, there exists an integer  $m_i \ge 1$ such that $X_i^{m_i} \in I$ and $X_i^{m_i-1} \not\in I$, 
and these integers satisfy 
\begin{equation}  \label{eq:expo}
m_1+\cdots + m_n + 1 - n \le K \le m_1m_2 \cdots m_n, 
\end{equation}
where the lower bound is exactly the number of
polynomials of the form $X_i^{j}$, $i=1,\ldots, n$
and $j=0,\ldots, m_i-1$. Clearly, $m_i \le K$ for each $i=1,\ldots,n$.

Denote by $S_K$ the set of $n$-tuples $\pmb{m}=(m_1,\ldots,m_n)$ of positive
integers satisfying the condition \eqref{eq:expo}.
Obviously, we have 
\begin{equation}  \label{eq:SK}
|S_K| \le K^n.  
\end{equation}
For each tuple $\pmb{m} \in S_K$, let $M_{n,\pmb{m}}(K)$ be the set of monomial ideals $I \in M_n(K)$ such that 
$X_i^{m_i} \in I$ and $X_i^{m_i-1} \not\in I$ for each $i=1,\ldots, n$. 
Then, we have 
\begin{equation}   \label{eq:MN}
M_n(K) = \bigcup_{\pmb{m} \in S_K}  M_{n,\pmb{m}}(K). 
\end{equation}
So, it suffices to estimate the size of each set $M_{n,\pmb{m}}(K), \pmb{m} \in S_K$. 

Now, fixing an $n$-tuple $\pmb{m}=(m_1,\ldots,m_n) \in S_K$, we want to estimate $|M_{n,\pmb{m}}(K)|$.
For each monomial ideal $I \in M_{n,\pmb{m}}(K)$, let $B(I)$ be the \textit{monomial basis} of the quotient space $\F[X_1,\ldots,X_n] / I$ over $\F$. 
That is, $B(I)$ is the set of monomials not contained in $I$, and $|B(I)|=K$. 
Note that every element in $B(I)$ is of the form $X_1^{j_1}X_2^{j_2}\cdots X_n^{j_n}$ 
with $0 \le j_i \le m_i-1$ for each $i=1,\ldots, n$. 

Since each monomial ideal $I \in M_{n,\pmb{m}}(K)$ is uniquely determined by $B(I)$, it is equivalent to estimate the possibilities of $B(I)$. 
Denote 
$$
D=m_1 + \cdots + m_n -n. 
$$ 
Since $K \ge 2$, we must have $D \ge 1$.   
By assumption, the monomials in $B(I)$ are of degree between 0 and $D$, and $|B(I)| = K$. 
Noticing $1 \in B(I)$, to obtain a possible choice for $B(I)$ we need to choose $k_i$ monomials of degree $i$ ($k_i\ge 0$) for each $i=1,2,\ldots,D$ 
such that 
$$
k_1 + k_2 + \cdots + k_D=|B(I)|-1=K-1.
$$ 
So, we obtain 
\begin{equation}   \label{eq:Mnm1}
|M_{n,\pmb{m}}(K)| \le \sum_{k_1 + \cdots + k_D=K-1} \prod_{d=1}^{D} \binom{N_d}{k_d}, 
\end{equation}
where $N_d$ is the number of monomials in $\F[X_1,\ldots,X_n]$ of degree $d$.  

It is well-known that
$$
N_d=\binom{n+d-1}{n-1}.
$$
Then, \eqref{eq:Mnm1} becomes
\begin{equation*}    \label{eq:Mnm2}
\begin{split}
|M_{n,\pmb{m}}(K)|  & \le \sum_{k_1 + \cdots + k_D=K-1} \prod_{d=1}^{D} N_d^{k_d}  \\
& \le \sum_{k_1 + \cdots + k_D=K-1} \prod_{d=1}^{D} (n+d-1)^{k_d(n-1)} \\
& \le (n+D-1)^{(n-1)(K-1)} \sum_{k_1 + \cdots + k_D=K-1} 1 \\
& =  (n+D-1)^{(n-1)(K-1)} \binom{D+K-2}{D-1} \\
& \le  (n+D-1)^{(n-1)(K-1)} (D+K-2)^{D-1}. 
\end{split}
\end{equation*}
By the definition of $D$ and \eqref{eq:expo},  we have $D \le K-1$, implying
\begin{equation*}
|M_{n,\pmb{m}}(K)| \le (K+n-2)^{(n-1)(K-1)} (2K-3)^{K-2}, 
\end{equation*}
which, together with  \eqref{eq:SK} and  \eqref{eq:MN}, gives the first part of the desired result. 

Finally, we want to get a simple upper bound for $|M_n(K)|$. 
We first have 
\begin{equation*}
|M_{n}(K)| \le K^n(K+n-2)^{(n-1)(K-1)} (2K-3)^{K-2} \le c_n^K K^{nK+n}, 
\end{equation*}
where $c_n$ is some constant depending only on $n$. 
Note that given a real number $R >1$ ($n$ is fixed), for any sufficiently large integer $K$ we have 
\begin{equation*}
c_n^K K^{nK+n} \le R^{K^{3/2}}. 
\end{equation*}
Hence, there exists a constant $c$ depending only on $n,R$ such that for any $K \ge 1$, 
$$
|M_n(K)| \le c_n^K K^{nK+n} \le cR^{K^{3/2}}.
$$
This completes the proof. 
\end{proof}

The estimate in the previous lemma might be not tight, but is sufficient for our purpose.

\begin{lemma}   \label{lem:coeff}
Let $I$ be a non-trivial ideal of $\F[X_1,\ldots,X_n]$ and
$G(I)$ be its reduced Gr{\"o}bner basis  with respect to the graded
lexicographic order.
If  the dimension of the quotient space
$\F[X_1,\ldots,X_n]/I$ over $\F$ equals
$K$,
then
$$
|G(I)| \le (n-1)K + 1,
$$
and the equality holds if and only if $K=1$ or $n=1$.
\end{lemma}

\begin{proof} 
First, note that the dimension of the quotient space $\F[X_1,\ldots,X_n]/I$ over $\F$ 
equals that of the quotient by the ideal generated by the leading monomials of
the polynomials in $G(I)$. 
So, without loss of generality we can assume that $I$ is a monomial ideal. 
Then, $G(I)$ is a set of monomials. 

Clearly, the equality holds when $n=1$. 
For $K=1$, we have $G(I) =  \{X_1,\ldots, X_n\}$, and so the inequality is indeed an equality. 
For fixed $n \ge 2$, the assertion 
$$
|G(I)| < (n-1)K + 1
$$ 
is proven by induction on $K \ge 2$.

For $K=2$,  the monomial basis for $\F[X_1,\ldots,X_n]/I$ over $\F$ is
$ \{1,X_i\}$ and $G(I)= \{X_1,\ldots, X_{i-1},X_{i}^2,X_{i+1},\ldots,
X_n\}$ for some $1\le i \le n$. So, 
$$
n = |G(I)| < 2(n-1) + 1 = 2n-1.
$$

Now, for general $K\ge 3$, take the maximal element (with respect to the graded lexicographical order)
 of the monomial basis of the quotient space $\F[X_1,\ldots,X_n]/I$ over $\F$: 
$$
\pmb{X^j} = X_1^{j_1} \ldots X_n^{j_n}.
$$
Consider the monomial ideal $J$ generated by $\pmb{X^j}$ and the monomials in $G(I)$. 
Note that $J$ is also a monomial ideal. 
First, for the reduced Gr{\"o}bner basis $G(J)$ of $J$,  noticing the choice of $\pmb{X^j}$  we have 
$$
G(J) = \{\pmb{X^j}\} \cup \big(G(I) \setminus \{X_1\pmb{X^j},\ldots, X_n\pmb{X^j}\} \big), 
$$
which implies that  
\begin{equation}   \label{eq:GIJ}
|G(J)| \ge |G(I)| -n + 1.
\end{equation}
Besides, the dimension of the quotient space $\F[X_1,\ldots,X_n]/J$ over $\F$ 
is $K-1$, because its monomial basis is obtained from that of $\F[X_1,\ldots,X_n]/I$
 by removing $\pmb{X^j}$.

Then, by induction hypothesis, we have 
$$
|G(J)| < (n-1)(K-1) +1, 
$$ 
which, together with \eqref{eq:GIJ}, gives 
\begin{equation*}
|G(I)| \le |G(J)| + n -1 < (n-1)(K-1) + n = (n-1)K + 1. 
\end{equation*}
This finishes the proof.
\end{proof}

\section{Lower and upper bounds for the linear complexity}
\label{sec:lower}

In this section, some  bounds for the linear complexity are presented,
which are analogues of the results in \cite[Theorems 1, 2 and 3]{GMS2}.

\subsection{Upper bounds}

The following result~\cite[Theorem 2]{Hoeffding} gives a dispersion
measure for the arithmetic  mean of equidistributed
random variables.

\begin{lemma}[Hoeffding's inequality]
  \label{lemma:hoefding}
  Let $X_1,\ldots, X_{d}$ be $d$ independent random variables
  with the same probability distribution, each ranging over the real
  interval $[a,b]$, and let $\mu$ be the expected value of each of
  these random variables. Then, for any $\epsilon>0$, the probability
  \begin{equation*}
    \mathcal{P} \left(
          \sum_{i=1}^{d}X_i \ge d(\mu + \epsilon) \right)
    \le \exp\Big(\frac{-2d\epsilon^2}{(b-a)^2}\Big).
  \end{equation*}
\end{lemma}

Using this lemma, we derive an upper bound, valid with large probability,
on the $k$-error linear complexity of periodic sequences.

\begin{theorem}
Let $k$ be a non-negative integer. 
  Let $\mu$ be the expected value of the linear complexity of a
  $T_1$-periodic sequence under the uniform probability distribution.
  Then, for any $\epsilon>0$, 
    choosing each periodic sequence of period $(T_1,\ldots,T_n)$ $s:
  \N_0^n \to \F_q$ with equal probability $1/q^{T_1\cdots T_n}$, we have
  the probability 
   \begin{equation*}
    \mathcal{P} \Big(  L_k(s) <  (\mu+\epsilon) T_2\cdots T_n \Big) >
    1- \exp \left(-2\epsilon^2T_2\cdots T_n /T_1^2\right), 
  \end{equation*}
\end{theorem}

\begin{proof}
  This result is a direct consequence of Lemma~\ref{lemma:hoefding},
  taking as random variables the linear complexities of the one-dimensional sequences 
  obtained from $(T_1,\ldots,T_n)$-periodic sequences by  fixing all coordinates  but the first one.

  More precisely, we denote by $L^{m_2,\ldots, m_n}(s)$ the linear complexity of
  the one-dimensional sequence defined by
  $\tilde{s}(m) = s(m,m_2,\ldots, m_n)$, where $m_2,\ldots,m_n$ are fixed.
  By definition, the following inequality holds:  
  $$
   L_k(s) \le L_0(s)= L(s) \le  \sum_{m_2=0}^{T_2-1}\ldots
  \sum_{m_n=0}^{T_n-1}L^{m_2,\ldots, m_n}(s). 
  $$
  Then, it suffices to show that
  $$
  \mathcal{P}\left(  \sum_{m_2=0}^{T_2-1}\ldots
  \sum_{m_n=0}^{T_n-1}L^{m_2,\ldots, m_n}(s) \ge (\mu+\epsilon)T_2\cdots T_n  \right) \le
  \exp \left(-2\epsilon^2T_2\cdots T_n/T_1^2 \right).
  $$
  However, the inequality above follows directly from Lemma
\ref{lemma:hoefding} by noticing that
each $L^{m_2,\ldots, m_n}(s)$ takes values in the interval $[0,T_1]$
and has expected value $\mu$.
\end{proof}

We remark that
there is an available formula in~\cite[Theorem 1]{Meidl} for the
expected value of the linear complexity of one-dimensional periodic sequences.

\subsection{Lower bounds}
This subsection is devoted to prove
some lower bounds for the linear complexity of
multidimensional periodic sequences valid  with large probability.
Although our results are not as strong as those by Gyarmati et al. \cite[Theorems 1 and 3]{GMS2},
they still suggest that  the expected value of linear complexity shall be large. 

\begin{theorem}
  \label{thm:L(s)}
  For any $\epsilon_1,\epsilon_2>0$, there is a constant
  $C(\epsilon_1,\epsilon_2,n,q)$
  such that, if an $n$-tuple
  $(T_1,\ldots,T_n)$ of positive integers satisfies $T_1\cdots T_n > C$,
  then choosing each periodic sequence of period $(T_1,\ldots,T_n)$ $s:
  \N_0^n \to \F_q$ with equal probability $1/q^{T_1\cdots T_n}$, we have
  the probability
\begin{equation}
\label{eq:event1}
\mathcal{P} \left( L(s) > \sqrt{ (1-\epsilon_1)T_1\cdots T_n/(n-1)}
\right) > 1- \epsilon_2.
\end{equation}
\end{theorem}

\begin{proof}
For any integer $K \ge 0$, let $S_{T_1,\ldots,T_n}(K)$ be the set of periodic
  sequences  with period $(T_1,\ldots,T_n)$ and with linear complexity $K$. 
  For our purpose, we need to estimate the size of $S_{T_1,\ldots,T_n}(K)$. 

For any sequence $s \in S_{T_1,\ldots,T_n}(K)$, 
  let  $G(s)$ be the reduced Gr{\"o}bner  basis of the ideal $I(s)$ with
  respect to the graded lexicographic order of $\F_q[X_1,\ldots,X_n]$. 
By definition, we know that $s$ can be generated by $K$ initial terms and using the linear recurrences represented by the polynomials in $G(s)$. 
So, we have 
\begin{equation}  \label{eq:STK}
|S_{T_1,\ldots,T_n}(K)| \le q^K \cdot \big( \textrm{the total number of possible choices of $G(s)$} \big). 
\end{equation}  

We first estimate the possibilities of the leading monomials of $G(s)$. 
  Let $J(s)$ be the monomial ideal generated by the leading monomials of $G(s)$.  
   Note that the dimension of the quotient space $\F_q[X_1,\ldots,X_n]/J(s)$ over $\F_q$
  is exactly $L(s)$, that is, $K$.
  By Lemma \ref{lem:monomial}, the number of possibilities of the monomial ideal
  $J(s)$ is at most 
  \begin{equation}   \label{eq:ideal}
    cq^{K^{3/2}}, 
  \end{equation}
  which is also an upper bound for the number of possibilities of the leading monomials of $G(s)$.  
  
Now, fixing the leading monomials of $G(s)$, we count the possibilities of $G(s)$. 
Note that the Delta set $\Delta(s)$ is also fixed, and $|\Delta(s)|=L(s)=K$ as indicated in \eqref{eq:LD}. 
Moreover, for each polynomial in $G(s)$, its non-leading monomials are of the form $\pmb{X^j},\pmb{j} \in \Delta(s)$. 
So, noticing all the polynomials in $G(s)$ are monic, we have that the number of possibilities of each polynomial in $G(s)$ is at most $q^{|\Delta(s)|} = q^K$. 
Then, using Lemma~\ref{lem:coeff}, the number of possibilities of $G(s)$, with fixed leading monomials, is at most 
   \begin{equation}    \label{eq:coefficient}
     q^{K|G(s)|} \le q^{(n-1)K^2 + K}.
   \end{equation}
   
Hence,  combining \eqref{eq:ideal} with \eqref{eq:coefficient}, the total number
  of the possibilities of $G(s)$ is at most 
  \begin{equation*}   
    cq^{K^{3/2}}q^{(n-1)K^2 + K} = cq^{(n-1)K^2+K^{3/2}+K}, 
  \end{equation*}
which, together with \eqref{eq:STK}, implies that 
 \begin{equation}  \label{eq:STK2}
|S_{T_1,\ldots,T_n}(K)| \le cq^{(n-1)K^2+K^{3/2}+2K}.
\end{equation}  

Now, we are ready to prove the claimed probability. 
    Write
  $$
  H = \lfloor \sqrt{ (1-\epsilon_1)T_1\cdots T_n/(n-1)} \rfloor.
  $$
  If the event considered in \eqref{eq:event1} does not hold for some
  sequence $s$, then there is an integer $K \le H$ such that
  $L(s)=K$.
  Thus, using \eqref{eq:STK2} we deduce that
  \begin{align*}
    \mathcal{P}(L(s) \le H) & = \frac{1}{q^{T_1\cdots T_n}} \sum_{K=0}^{H} |S_{T_1,\ldots,T_n}(K)|    \\
                            & \le \frac{1}{q^{T_1\cdots T_n}} \sum_{K=0}^{H}  cq^{(n-1)K^2+K^{3/2}+2K}
                              \le cq^{(n-1)H^2+H^{3/2}+3H-T_1\cdots T_n}.
  \end{align*}
  So, for large enough $T_1\cdots T_n$, we have
  \begin{align*}
    \mathcal{P}\left( L(s) \le H \right) & \le cq^{((1-\epsilon_1)T_1\cdots T_n/(n-1))^{3/4}+3\sqrt{(1-\epsilon_1)T_1\cdots T_n/(n-1)}-\epsilon_1T_1\cdots T_n} \\
                                         & \le \epsilon_2.
  \end{align*}
This in fact completes the proof.
\end{proof}

Similarly, we can get a lower bound for the $k$-error linear
complexity of  multidimensional periodic sequences with large
probability.

\begin{theorem}
  \label{thm:Lk(s)}
  For any $\epsilon_1,\epsilon_2>0$, there are  numbers
  $\epsilon_3(\epsilon_1,q)$ and $C(\epsilon_1,
  \epsilon_2,\epsilon_3,n,q)$
  such that if an $n$-tuple $(T_1,\ldots,T_n)$ of positive integers satisfies $T_1\cdots T_n >C$
  and a non-negative integer $k$ satisfies $k<\epsilon_3 T_1\cdots T_n$,
  then choosing each periodic sequence of period $(T_1,\ldots,T_n)$ $s:
  \N_0^n \to \F_q$ with equal probability $1/q^{T_1\cdots T_n}$, we
  have the probability 
  \begin{equation*}
    \mathcal{P} \left( L_k(s) > \sqrt{ (1-\epsilon_1)T_1\cdots
        T_n/(n-1)} \right) > 1- \epsilon_2.
  \end{equation*}
\end{theorem}

\begin{proof} 
Let 
$$
H = \lfloor \sqrt{ (1-\epsilon_1)T_1\cdots T_n/(n-1)} \rfloor.
$$
Denote by $W_{T_1,\ldots,T_n}(k,H)$ the set of $(T_1,\ldots,T_n)$-periodic
  sequences with $k$-error linear complexity at most $H$. 
 First we need to estimate the size of the set $W_{T_1,\ldots,T_n}(k,H)$. 

For two periodic sequences $s$ and $\sigma$ of period $(T_1,\ldots,T_n)$, define
$$
d(s,\sigma) = |\{\pmb{m}=(m_1,\ldots,m_n):\, s(\pmb{m})\ne
\sigma(\pmb{m}), 0\le m_i \le T_i-1, 1\le i \le n\}|.
$$
By definition, for a periodic sequence $s \in W_{T_1,\ldots,T_n}(H)$,
 there is another $(T_1,\ldots,T_n)$-periodic sequence $\sigma$ having the linear complexity
$$
L(\sigma) =L_k(s) \le H
$$
such that
\begin{equation}
\label{eq:distance}
d(s,\sigma)\le k.
\end{equation}
Conversely, if  a $(T_1,\ldots,T_n)$-periodic sequence $\sigma$ is fixed, then
a periodic sequence $s$ satisfying \eqref{eq:distance} can be obtained
from $\sigma$ by changing $\sigma(m_1,\ldots,m_n)$ for at most $k$ of
the $T_1\cdots T_n$ tuples $(m_1,\ldots,m_n),0\le m_i \le T_i-1, 1\le
i \le n$. This yields at most $q^k\binom{T_1\cdots T_n}{k}$ 
sequences  from each $\sigma$.
So, we obtain 
\begin{equation}
\label{eq:num-Lk}
\begin{split}
& |W_{T_1,\ldots,T_n}(k,H)| \\
& \quad \le |\{\textrm{ sequence $\sigma$ of period $(T_1,\ldots,T_n)$\,:\, $L(\sigma)\le H$ }\}| \cdot q^k\binom{T_1\cdots T_n}{k}.
\end{split}
\end{equation}
Moreover,  if $\epsilon_3$ is small enough in terms of $\epsilon_1$ and $q$, and
$T_1\cdots T_n$ is large enough in terms of $\epsilon_1,
\epsilon_3$ and $q$, then it follows from $k<\epsilon_3 T_1\cdots T_n$ that
\begin{equation}   \label{eq:binom}
q^k\binom{T_1\cdots T_n}{k} < 2^{\epsilon_1T_1\cdots T_n /2}.
\end{equation}

Then, combining \eqref{eq:num-Lk} with \eqref{eq:STK2} and \eqref{eq:binom}, we obtain
\begin{align*}
 |W_{T_1,\ldots,T_n}(k,H)|  & \le \sum_{K=0}^{H} 2^{\epsilon_1T_1\cdots T_n /2} |S_{T_1,\ldots,T_n}(K)|  \\
& \le 2^{\epsilon_1T_1\cdots T_n /2}\sum_{K=0}^{H} cq^{(n-1)K^2+K^{3/2}+2K}  \\
& \le cq^{(n-1)H^2 + H^{3/2} + 3H + \epsilon_1T_1\cdots T_n /2},
\end{align*}
where $c$ is some absolute constant depending on $n, q$.  
Therefore, for  $T_1\cdots T_n$ large enough, we have 
\begin{align*}
\mathcal{P} \big( L_k(s) \le H \big) & =  \frac{|W_{T_1,\ldots,T_n}(k,H)|}{q^{T_1 \cdots T_n}}  \\ 
&  \le cq^{((1-\epsilon_1)T_1\cdots T_n/(n-1))^{3/4}+3\sqrt{(1-\epsilon_1)T_1\cdots
                                         T_n} -\epsilon_1T_1\cdots T_n/2}   \\
  &  \le \epsilon_2,
\end{align*}
which in fact completes the proof.
\end{proof}

\section*{Acknowledgements}

The first and second authors are grateful to the Centre International de Rencontres Math{\'e}matiques in Luminy for hosting them
in a Research in Pairs program, where this paper was initiated.

D. G{\'o}mez-P{\'e}rez was partially supported
by the project MTM2014-55421-P from the Spanish Ministerio de Economia y Competitividad.  
The research of M. Sha was supported by the Australian Research Council Grant DP130100237,
and he was also supported by the Macquarie University Research Fellowship.


\end{document}